\documentclass{amsart}
\usepackage{amsthm, amssymb,amsmath,amsfonts,amscd, mathrsfs,mathtools,color, enumerate,fullpage, thmtools, tikz-cd, comment}
\usepackage{float}
\usepackage{wrapfig}
\restylefloat{figure}

\usepackage[pagebackref]{hyperref}
\usepackage[capitalize, noabbrev]{cleveref}
\usepackage[all]{xy}
\usepackage{latexsym}

\newtheorem{definition}{Definition}[section]
\newtheorem{lemma}[definition]{Lemma}
\newtheorem{proposition}[definition]{Proposition}
\newtheorem{corollary}[definition]{Corollary}

\newtheorem{theorem}[definition]{Theorem}
\newtheorem{conjecture}[definition]{Conjecture}
\newtheorem*{notation}{Notation}

\theoremstyle{remark}

\declaretheorem[name=Remark,sibling=theorem,qed={\lower-0.3ex\hbox{$\diamond$}}]{remark}
\declaretheorem[name=Note,sibling=theorem,qed={\lower-0.3ex\hbox{$\diamond$}}]{note}

\title[Algebraicity of $L$-values]
{Algebraicity of $L$-values for $\GSp_4 \times \GL_2$ and $\GSp_4 \times \GL_2 \times \GL_2$}

\newcommand{\fa}{\mathfrak{a}}
\newcommand{\fb}{\mathfrak{b}}
\newcommand{\fz}{\mathfrak{z}}

\DeclareMathOperator{\FL}{FL}
\DeclareMathOperator{\Iw}{Iw}
\newcommand{\ZZ}{\mathbf{Z}}
\newcommand{\QQ}{\mathbf{Q}}
\newcommand{\RR}{\mathbf{R}}
\newcommand{\CC}{\mathbf{C}}

\newcommand{\GG}{\mathbf{G}}

\renewcommand{\AA}{\mathbf{A}}
\newcommand{\Qp}{\QQ_p}

\newcommand{\Zp}{\ZZ_p}

\newcommand{\cW}{\mathcal{W}}
\newcommand{\cV}{\mathcal{V}}

\renewcommand{\ge}{\geqslant}
\renewcommand{\le}{\leqslant}
\renewcommand{\leq}{\leqslant}
\renewcommand{\geq}{\geqslant}

\DeclareMathOperator{\dep}{dep}
\DeclareMathOperator{\diag}{diag}

\DeclareMathOperator{\SO}{SO}

\DeclareMathOperator{\fin}{f}

\DeclareMathOperator{\GL}{GL}

\DeclareMathOperator{\GSp}{GSp}

\DeclareMathOperator{\Hom}{Hom}

\DeclareMathOperator{\Kl}{Kl}

\DeclareMathOperator{\Si}{Si}
\DeclareMathOperator{\SL}{SL}
\DeclareMathOperator{\Sp}{Sp}

\DeclareMathOperator{\sph}{sph}

\DeclareMathOperator{\tor}{tor}

\newenvironment{smatrix}{\left(\begin{smallmatrix}}{\end{smallmatrix}\right)}

\usepackage{tabto}

\newcommand{\tbt}[4]{
	\ensuremath{
		\begin{pmatrix}
			#1 & #2 \\
			#3 & #4
		\end{pmatrix}
	}
}
\newcommand{\stbt}[4]{
	\ensuremath{
		\begin{smatrix}
			#1 & #2 \\
			#3 & #4
		\end{smatrix}
	}
}

\begin{document}

\begin{abstract}
 We prove algebraicity results for critical $L$-values attached to the group $\GSp_4 \times \GL_2$, and for Gan--Gross--Prasad periods which are conjecturally related to central $L$-values for $\GSp_4 \times \GL_2 \times \GL_2$. Our result for $\GSp_4 \times \GL_2$ gives a new proof (by a very different method) of a recent result of Morimoto, and will be used in a sequel paper to construct a new $p$-adic $L$-function for $\GSp_4 \times \GL_2$. The results for Gross--Prasad periods appear to be new.  A key aspect is the computation of certain archimedean zeta integrals, whose $p$-adic counterparts are also studied in this note.
\end{abstract}

\author{David Loeffler}
\address{D.L.: Mathematics Institute, University of Warwick, Coventry CV4 7AL, United Kingdom}
\email{D.A.Loeffler@warwick.ac.uk}

\author{\'Oscar Rivero}
\address{O.R.: Mathematics Institute, University of Warwick, Coventry CV4 7AL, United Kingdom}
\email{Oscar.Rivero-Salgado@warwick.ac.uk}

\thanks{Supported by ERC Consolidator grant ``Shimura varieties and the BSD conjecture'' (D.L.) and Royal Society Newton International Fellowship NIF\textbackslash R1\textbackslash 202208 (O.R.). This material is based upon work supported by the National Science Foundation under Grant No. DMS-1928930 while the authors were in residence at the Mathematical Sciences Research Institute in Berkeley, California, during the Spring 2023 semester.}


 \maketitle
 \setcounter{tocdepth}{1}
 \tableofcontents

\section{Introduction}

 \subsection{Critical $L$-values for $\GSp_4 \times \GL_2$}
  \label{sect:introEis}

  Let $\pi \times \sigma$ be an automorphic representation of $\GSp_4 \times \GL_2$. We can attach to $\pi$ and $\sigma$ a degree 8 $L$-function $L(\pi \times \sigma, s)$, associated to the tensor product of the natural degree 4 (spin) and degree 2 (standard) representations of the $L$-groups of $\GSp_4$ and $\GL_2$.

  If $\pi$ and $\sigma$ are algebraic, then this $L$-function is expected to correspond to a motive, and in particular we can ask whether it has critical values. More precisely, suppose that (the $L$-packet of) $\pi$ corresponds to a holomorphic Siegel modular form of weight $(k_1, k_2)$, with $k_1 \ge k_2 \ge 2$; and that $\sigma$ corresponds to a modular form of weight $\ell \ge 1$. Then we expect there to exist motives $M(\pi)$ (of motivic weight $k_1 + k_2 - 3$) and $M(\sigma)$ (of motivic weight $\ell -  1$) such that
  \[
   L(\pi \times \sigma, s) = L\left(M(\pi) \otimes M(\sigma), s + \tfrac{w}{2}\right), \qquad w = k_1 + k_2 + \ell - 4.
  \]

  For $L(\pi \times \sigma, s)$ to be a critical value, we must have $s = -\tfrac{w}{2} \bmod \ZZ$, and $s$ must satisfy various inequalities depending on how the weights of $\pi$ and $\sigma$ interlace. For fixed $(k_1, k_2)$, the allowed pairs $(s, \ell)$ form three disjoint polygonal regions in the plane (all symmetric about the line $s = \tfrac{1}{2}$); for compatibility with the conjectures of \cite{LZvista} on Gross--Prasad periods (see below), we denote these by $(A)$, $(D)$ and $(F)$. The inequalities defining these regions are given in \cref{table1}.

  The first main goal of this paper is to prove an algebraicity result for weights in region $(D)$. As we shall shortly recall, many cases of this algebraicity are already known by work of Morimoto, B\"ocherer--Heim and Saha. However, our method is independent of these works and uses rather different methods; it gives some cases of algebraicity which are not already known, and (perhaps more importantly) will serve as the starting-point for results on interpolation of region $(D)$ $L$-values in $p$-adic families, which will be treated in a sequel paper.

  \begin{table}[ht]
   \caption{\label{table1}Critical regions for $\GSp_4 \times \GL_2$}
   \begin{tabular}{c r@{\hskip 0.3em}c@{\hskip 0.3em}l @{\hskip 2em} l}
    Region & \multicolumn{3}{c}{Range of $\ell$} & Range of $s$ \\[0.25ex]
    \hline
    \rule{0pt}{2.5ex}$(A)$ & $k_1 + k_2 - 1$ & $\le \ell \phantom{\le}$ & &
    $|2s-1| \le \ell - (k_1 + k_2 - 1)$\\[0.25ex]
    $(D)$ & $k_1 - k_2 + 3$ &$\le \ell \le$& $k_1 + k_2 - 3$ &
    $|2s-1| \le \min\left(k_1 + k_2 - 3 - \ell, \ell - (k_1 - k_2 + 3)\right)$\\[0.25ex]
    $(F)$ & $1$ & $\le \ell \le$ & $k_1 - k_2 + 1$ & $|2s-1| \le k_1 - k_2 + 1 - \ell$
   \end{tabular}
   \medskip

   (In the ``missing'' cases $\ell = k_1 \pm (k_2 - 2)$, the $L$-function has no critical values.)
  \end{table}

 \subsection{Gross--Prasad periods for $\GSp_4 \times \GL_2 \times \GL_2$}
  \label{sect:introcusp}

  The second problem we consider here is to study the \emph{Gross--Prasad period} attached to a cuspidal automorphic representation $\pi$ of $\GSp_4$ and two cuspidal automorphic representations $\sigma_1, \sigma_2$ of $\GL_2$, which is defined by
  \[ \mathcal{P}(\varphi, \psi_1, \psi_2) = \int_{[H]} \varphi(h) \psi_1(h_1) \psi_2(h_2) \mathrm{d}h, \]
  for forms $\varphi \in \pi$ and $\psi_i \in \sigma_i$. Here $H$ denotes the group $\GL_2 \times_{\GL_1} \GL_2$, and $[H]$ the adelic symmetric space $\RR^\times H(\QQ) \backslash H(\AA)$. (We may assume $\chi_{\pi} \chi_{\sigma_1} \chi_{\sigma_2} = 1$, since the integral is trivially zero otherwise.) The Gross--Prasad period vanishes unless the central characters satisfy the condition $\varepsilon(\pi_v \times \sigma_{1, v} \times \sigma_{2, v}) = +1$ for all places $v$. If this condition holds, the Gross--Prasad conjecture for $\SO_4 \times \SO_5$ predicts that $\mathcal{P}(-)$ is non-zero if and only if $L(\pi \times \sigma_1 \times \sigma_2, \tfrac{1}{2}) \ne 0$, and Ichino and Ikeda \cite{ichinoikeda10} have formulated a precise conjecture relating $|\mathcal{P}(\varphi, \psi_1, \psi_2)|^2$ to the central $L$-value. (More precisely, the original works of Gross--Prasad and Ichino--Ikeda apply when $\chi_\pi = 1$, so $\pi$ factors through $\SO_5$ and $\sigma = \sigma_1 \boxtimes \sigma_2$ through its subgroup $\SO_4$; the more general formulation we use here is due to Emory \cite{emory19}.) We shall prove here results for the periods $\mathcal{P}(-)$, which are not logically dependent on the conjectured relation to $L$-values, but nonetheless the Gross--Prasad and Ichino--Ikeda conjectures are the key motivation for studying these periods.

  If we take $\pi$ to have weight $(k_1, k_2)$, as above, and $\sigma_1, \sigma_2$ to correspond to holomorphic cuspforms of weights $c_1, c_2$, then we have 9 different cases $\{a, a', b, b', c, d, d', e, f\}$ depending on the inequalities satisfied by $(c_1, c_2)$ and $(k_1, k_2)$. See Figure 2 of \cite{LZvista} for a diagram illustrating these. In cases $\{b,b', e\}$ the Gross--Prasad period is automatically zero. In the remaining six cases one expects a rationality result for the Gross--Prasad period, whose formulation will depend on which case we consider; this is implicit in Conjecture 4.1.2 of \cite{LZvista}.

  The relation between these periods and the $\GSp_4 \times \GL_2$ problem above is via Novodvorsky's integral formula for the $\GSp_4 \times \GL_2$ $L$-function \cite{novodvorsky79}. Novodvorsky's integral can be seen as a ``degenerate case'' of the Gross--Prasad period, in which the cuspidal $\GL_2$-representation on the first factor of $H$ is replaced by a space of Eisenstein series. This relation allows the two rationality problems to be treated in parallel. If $s$ is a critical value for the $\GSp_4 \times \GL_2$ $L$-function, then the Eisenstein automorphic representation playing the role of $\sigma_1$ has weight $1 + |2s - 1|$; and the cases $(A)$, $(D)$, $(F)$ of the previous section correspond to requiring that the pair $(c_1, c_2) = (1 + |2s - 1|, \ell)$ should satisfy the inequalities $(a), (d), (f)$ respectively.

 \subsection{Known results for $\GSp_4 \times \GL_2$}

  Let $\pi$, $\sigma$ be as in \cref{sect:introEis}. We assume that $\ell \ne k_1 \pm (k_2 - 2)$, so that critical values exist and we are in one of the cases $(A)$, $(D)$, $(F)$. We may suppose without loss of generality that $\pi$ is \emph{tempered}, since Arthur's classification of the discrete spectrum \cite{arthur04, geetaibi18} shows that for non-tempered $\pi$, the $L$-function can be expressed as a product of automorphic $L$-functions for $\GL_2$ and $\GL_2 \times \GL_2$, and the algebraicity properties of these $L$-values are well-understood.

  Let $m$ denote the sum of the four smallest Hodge numbers of $M(\pi) \otimes M(\sigma)$, so that
  \[ m =
   \begin{cases}
    2k_1 + 2k_2 - 6 & \text{case $(A)$},\\
    k_1 + k_2 + \ell - 4 & \text{case $(D)$},\\
    2k_2 + 2\ell - 6 & \text{case $(F)$}.
   \end{cases}
  \]

  \begin{conjecture}\label{conj:algebraicity}
   \begin{enumerate}[(1)]
    \item There exists a period $\Omega(\pi \times \sigma) \in \CC^\times$ with the following property: for every $j \in \ZZ$ such that $s = -\tfrac{w}{2} + j$ is a critical value, we have
    \[
     \frac{L(\pi \times \sigma, -\tfrac{w}{2} + j)}{(-2\pi i)^{4j-m} \Omega(\pi \times \sigma)}
     \in \overline{\QQ}.
    \]

    \item For $\chi$ a Dirichlet character we have $\Omega(\pi \times \sigma \times \chi) = \Omega(\pi \times \sigma)$ mod $\overline{\QQ}^\times$.

   \end{enumerate}
  \end{conjecture}

  Of course, one would also like to have a tolerably explicit form for the period $\Omega(\pi \times \sigma)$. For cases (A), (D), Morimoto \cite[Proposition 2.1]{morimotoAdvances} has shown that Deligne's general algebraicity conjectures \cite{deligne79} imply an explicit form for the periods $\Omega(\pi \times \sigma)$ in terms of Petersson norms of holomorphic eigenforms.

  \subsubsection{Known results: case $(F)$}

   For case $F$, the conjecture is known in full: it is proved in \cite{LPSZ1}, building on Harris' study of ``occult periods'' for the degree 4 $L$-function of $\GSp_4$ in \cite{harris04}. In this case, the period $\Omega(\pi \times \sigma)$ depends only on $\pi$ (not on $\sigma$), and is defined using the $\overline{\QQ}$-structure on $H^2$ of coherent automorphic sheaves on a toroidal compactification of the Siegel modular threefold.

  \subsubsection{Known results: case $(D)$}

   There are a number of papers in the literature establishing Conjecture \ref{conj:algebraicity} for weights in region (D) under various additional hypotheses on the weights and/or levels.
   \begin{itemize}

    \item B\"ocherer--Heim \cite{BH06} consider the case $k_1 = k_2 = 2k_G$ and $\ell = 2k_h$ for some integers $k_G, k_h$, and $\pi$ and $\sigma$ are unramified at all finite primes, generated by holomorphic eigenforms $G$ for $\operatorname{Sp}_4(\ZZ)$ and $h$ for $\operatorname{SL}_2(\ZZ)$ respectively. For such $\pi$, $\sigma$, they prove part (1) of \cref{conj:algebraicity} for the explicit period
    \[ \Omega(\pi \times \sigma)^{\mathrm{BH}} \coloneqq \langle G, G \rangle \cdot \langle h, h \rangle,\]
    where $\langle -, -\rangle$ denotes the Petersson scalar product (and $G$ and $h$ are normalised to have their Fourier coefficients in $\overline{\QQ}$). Their result assumes that the first Fourier--Jacobi coefficient of $G$ is non-vanishing.

    \item Saha \cite{Saha} proves an analogous result with slightly stronger assumptions on the weight (he assumes that $k_1 = k_2 = \ell \ge 6$), but less restrictive conditions on the level (allowing $\pi$ and $\sigma$ to be either the Steinberg representation, or its unramified quadratic twist, at some finite places). This refines earlier works of Furusawa, Pitale--Schmidt, and Saha himself. As in B\"ocherer--Heim, Saha assumes a non-vanishing hypothesis for a certain Fourier coefficient (depending on the levels of $\pi$ and $\sigma$); and the period he uses is again $\langle G, G \rangle \cdot \langle h, h \rangle$.

    \item Morimoto \cite{morimotoAdvances} considers the more general case where $k_1$ need not be equal to $k_2$ (so $\pi$ is generated by vector-valued, rather than scalar-valued, holomorphic Siegel modular forms), and makes no assumption on the levels at all. However, his result excludes certain critical values close to the centre of the functional equation.
   \end{itemize}

  \subsubsection{Case $(A)$}

   For case $(A)$, the relevant period $\Omega(\pi \times \sigma)$ should be independent of $\pi$ and given by $\langle g, g \rangle^2$, where $g$ is the normalised newform generating $\sigma$ as before. This has been proved, under some auxiliary hypotheses on the weights and levels, by Furusawa--Morimoto \cite{furusawamorimoto}.

 \subsection{Our results}

  In this paper, we focus on case $(D)$, or more precisely on the following sub-case:

  \begin{definition}
   Say we are in case $(D^-)$ if $(k_1, k_2, \ell, s)$ satisfy the inequalities of case $(D)$ together with the additional condition $\ell \le k_1$.
  \end{definition}

  For weights in this region we prove the following:

  \begin{theorem}\label{thm:main}
   Assume the weights of $(\pi, \sigma)$ are in case $(D^-)$, and the following local hypothesis holds:
   \begin{itemize}
   \item For each prime $\ell$ (if any) such that $\sigma_\ell$ is supercuspidal, the central character of $\pi_\ell$ is a square in the group of characters of $\QQ_\ell^\times$.
   \end{itemize}
   Then \cref{conj:algebraicity} holds for $\pi \times \sigma$, with a period of the form
   $\Omega(\pi \times \sigma) = \Omega^{(1)}(\pi) \cdot \langle h, h \rangle$, where $\langle h, h \rangle$ denotes the Petersson norm of the normalised newform generating $\sigma$, and $ \Omega^{(1)}(\pi)$ is a period depending only on $\pi$ (defined using $H^1$ of coherent sheaves on the compactified Siegel threefold).
  \end{theorem}

  The local hypothesis is required in order to show that Novodvorsky's zeta integral computes the correct $L$-factors at the bad primes. It is, of course, vacuous if $\pi$ has trivial central character. We hope that future work will allow this assumption to be removed.

  Our methods are very different from the works cited above. Where those works use integral formulae involving holomorphic (or nearly-holomorphic) Siegel cusp forms, and Eisenstein series restricted from a much larger group (namely $U(3, 3$), we instead work with a non-holomorphic cusp form for $\GSp_4$ belonging to the unique \emph{globally generic} representation in the same $L$-packet as $\pi$, and an Eisenstein series pushed forward from $\GL_2 \times \GL_2$. In particular, our results are unconditional, not depending on any non-vanishing assumptions on Fourier coefficients; and they cover all critical values of the relevant $L$-function (even the near-central ones excluded in \cite{Saha} and \cite{morimotoAdvances}).

  \begin{remark}
   Comparing our results with those of Morimoto cited above, we can deduce a consequence for the arithmetic of $\pi$ alone: if $\pi$ has weight $(k_1, k_2)$ and $k_2$ is sufficiently large that Theorem 1 of \emph{op.cit.} applies for some $\ell$ (it suffices to take $k_2 > 6$), then the $H^1$ Whittaker period $\Omega^{(1)}(\pi)$ is a $\overline{\QQ}^\times$-multiple of $\langle G, G \rangle$, for an appropriately normalised Siegel eigenform $G$ generating $\pi$. This period relation is not at all obvious a priori.
  \end{remark}

  If $k_2 > 3$, then the set of $(s, \ell)$ satisfying $(D^-)$ is a strict subset of the $(s, \ell)$ satisfying $(D)$; and the methods used in this article do not apply for the remaining cases. It seems likely that these remaining cases will require a theory of Maass--Shimura-type differential operators acting on $H^1$ of Siegel threefolds; this is beyond the scope of the present work.

 \subsection{Connection with other works}

  In a sequel to this paper \cite{LR3}, we build on the algebraicity results developed here in order to define a $p$-adic $L$-function interpolating $L$-values along the ``edge'' of region $(D)$ (that is, with $s = \frac{k_1 + k_2 - 2 - \ell}{2}$).

  This computation requires the evaluation of a certain local zeta integral at the $p$-adic place (which gives the Euler factor $\mathcal E^{(D)}(\pi \times \sigma)$ appearing in the interpolation property of our $p$-adic $L$-function). We carry out this local computation here, rather than in the sequel paper, since its proof has much in common with the local Archimedean computation needed to prove \cref{thm:main} (and little in common with the $p$-adic interpolation computations which form the bulk of the sequel paper).

\subsection*{Acknowledgements} The authors would like to thank Sarah Zerbes for informative conversations related to this work, and Kazuki Morimoto for patiently explaining his results in \cite{morimotoAdvances} to us.

\section{The groups $G$ and $H$ and their Shimura varieties}

\subsection{Groups and parabolics}

We denote by $G$ the group scheme $\GSp_4$ (over $\ZZ$), defined with respect to the anti-diagonal matrix $J = \begin{smatrix}&&&1\\ &&1 \\ &-1 \\ -1 \end{smatrix}$; and we let $\nu$ be the multiplier map $G \to \GG_m$. We define $H = \GL_2 \times_{\GL_1} \GL_2$, which we embed into $G$ via the embedding
\[ \iota \, : \, \left[
\begin{pmatrix}a&b\\c&d\end{pmatrix}, \begin{pmatrix}a'&b'\\c'&d'\end{pmatrix} \right]
\mapsto \begin{pmatrix}a&&&b\\&a'&b'\\&c'&d'\\c&&&d\end{pmatrix}.
\]
We sometimes write $H_i$ for the $i$-th $\GL_2$ factor of $H$. We write $T$ for the diagonal torus of $G$, which is contained in $H$ and is a maximal torus in either $H$ or $G$.

We write $B_G$ for the upper-triangular Borel subgroup of $G$, and $P_{\Si}$ and $P_{\Kl}$ for the standard Siegel and Klingen parabolics containing $B$, so
\[ P_{\Si} =
\begin{smatrix}\star & \star & \star & \star \\ \star & \star & \star & \star\\
&& \star & \star\\ && \star & \star \end{smatrix}, \qquad
P_{\Kl} =
\begin{smatrix}\star & \star & \star & \star \\  & \star & \star & \star\\
&\star & \star & \star\\ &&& \star \end{smatrix}.
\]
We write $B_H = \iota^{-1}(B_G) = \iota^{-1}(P_{\Si})$ for the upper-triangular Borel of $H$.

In this paper $P_{\Si}$ will be much more important than $P_{\Kl}$ (in contrast to \cite{LPSZ1}). We have a Levi decomposition $P_{\Si} = M_{\Si} N_{\Si}$, with $M_{\Si} \cong \GL_2 \times \GL_1$, identified as a subgroup of $G$ via
\[
(A, u) \mapsto \begin{pmatrix}A\\&uA'\end{pmatrix},\qquad A' \coloneqq \begin{pmatrix}&1\\1\end{pmatrix} \cdot {}^t\!A^{-1} \cdot \begin{pmatrix}&1\\1\end{pmatrix}.
\]
The intersection $B_M \coloneqq M \cap B_G$ is the standard Borel of $M$; its Levi factor is $T$.

\subsection{Flag varieties and Bruhat cells}

We write $\FL_G$ for the Siegel flag variety $P \backslash G$, with its natural right $G$-action. There are four orbits for the Borel $B_G$ acting on $\FL_G$, the \emph{Bruhat cells}, represented by a subset of the Weyl group of $G$, the \emph{Kostant representatives}, which are the smallest-length representatives of the quotient $W_M \backslash W_G$. We denote these by $w_0, \dots, w_3$; see \cite{LZ21-erl} for explicit matrices. Note that the cell $C_{w_i} = P \backslash P w_i B_G \subset \FL_G$ has dimension $\ell(w_i) = i$.

Analogously, for the $H$-flag variety $\FL_H = B_H \backslash H$, we have 4 Kostant representatives $w_{00} = \mathrm{id}$, $w_{10} = \left(\begin{pmatrix}0&1\\-1&0\end{pmatrix}, \mathrm{id}\right)$, similarly $w_{01}$, $w_{11}$ (with the cell $C_{w_{ij}}$ having dimension $i + j$). (This is the whole of the Weyl group of $H$, since the Levi subgroup of $M_H = T$ is trivial.)

 \subsection{Representations}



  We retain the conventions about algebraic weights and roots of \cite{LZ21-erl}. In particular, we identify characters of $T$ with triples of integers $(r_1,r_2; c)$, with $r_1+r_2 = c$ modulo 2 corresponding to $\diag(st_1,st_2,st_2^{-1},st_1^{-1}) \mapsto t_1^{r_1} t_2^{r_2} s^c$. With our present choices of Borel subgroups, a weight $(r_1,r_2; c)$ is dominant for $H$ if $r_1,r_2 \geq 0$, dominant for $M_G$ if $r_1 \geq r_2$, and dominant for $G$ if both of these conditions hold. (We frequently omit the central character $c$ if it is not important in the context.)

 \subsection{Models of Shimura varieties}

  Let $K$ be a neat open compact subgroup of $\GSp_4(\AA_{\fin})$. Denote by $Y_{G,\QQ}$ the canonical model over $\QQ$ of the level $K$ Shimura variety. It is a smooth quasiprojective threefold, whose complex points are canonicallly identified with a double-coset space of $G(\AA)$, as discussed in \cite[\S2.3]{LPSZ1}. We write $Y_{H,\QQ}$ for the canonical $\QQ$-model of the Shimura variety for $H$ of level $K_H = K \cap H(\AA_{\fin})$, which is a moduli space for ordered pairs of elliptic curves with level structure. Further, there is a morphism of algebraic varieties $\iota \, : \, Y_{H,\QQ} \rightarrow Y_{G,\QQ}$.

  After choosing a suitable combinatorial datum (a rational polyhedral cone decomposition), we can define a smooth compactification $X_{G, \QQ}$ of $Y_{G, \QQ}$. This depends on the choice of cone decomposition, but we shall not indicate this in the notation, since the choice of cone decomposition will remain fixed throughout. As usual, we denote by $D$ the boundary divisor.

  The cone decomposition for $G$ naturally determines a cone decomposition for $H$ and hence a compactification $X_{H, \QQ}$ of $Y_{H, \QQ}$, and the embedding $\iota$ extends to a finite morphism $X_{H, \QQ} \to X_{G, \QQ}$ (which we also denote by $\iota$). One can in fact always choose the cone decomposition in such a way that this map of toroidal compactifications is a closed immersion \cite{lan19}, although we do not need this here.

  \begin{remark}
   If $K_H$ is the product of subgroups $K_{H, 1} \times K_{H, 2}$, then $Y_{H, \QQ}$ is a product of two modular curves. Each of these has a canonical compactification given by adjoining finitely many cusps; but $X_{H, \QQ}$ may not be the product of these compactified modular curves (depending on the choice of the toroidal boundary data). In general $X_{H, \QQ}$ will be obtained from the product of compactified modular curves by a finite sequence of blow-ups concentrated above points of the form $(\text{cusp}) \times (\text{cusp})$.
  \end{remark}

 \subsection{Coefficient sheaves}\label{subsec:coef}

  We adopt the conventions recalled in \cite[\S2.5.1]{LZ21-erl}. For our further use, recall that the Weyl group acts on the group of characters $X^*(T)$ via $(w \cdot \lambda)(t) = \lambda(w^{-1}tw)$. As discussed in loc.\,cit., we can define explicitly $w_G^{\max}$, the longest element of the Weyl group, as well as $\rho = (2, 1; 0)$, which is half the sum of the positive roots for $G$.

  There is a functor from representations of $P_G$ to vector bundles on $X_{G, \QQ}$; and we let $\cV_{\kappa}$, for $\kappa \in X^{\bullet}(T)$ that is $M_G$-dominant, be the image of the irreducible $M_G$-representation of highest weight $\kappa$. Given an integral weight $\nu \in X^\bullet(T)$ such that $\nu + \rho$ is dominant, we define
  \[ \kappa_i(\nu) = w_i(\nu+\rho) - \rho, \quad 0 \leq i \leq 3, \]
  where as usual $\rho$ is half the sum of the positive roots. These are the weights $\kappa$ such that representations of infinitesimal character $\nu^{\vee}+\rho$ contribute to $R\Gamma(S_K^{G,\tor},\cV_{\kappa})$; if $\nu$ is dominant (i.e. $r_1 \ge r_2 \ge 0$), they are the weights which appear in the \emph{dual BGG complex} computing de Rham cohomology with coefficients in the algebraic $G$-representation of highest weight $\nu$. See \cite[\S2.5.2]{LZ21-erl} for explicit formulae.





\subsection{Cohomology}
According to the previous discussion, if $V$ is an algebraic representation of $P_S$ over $\ZZ_{(p)}$, we have a vector bundle $\cV$ on $X_G$ defined by $\cV:= V \times^{P_S} \mathcal T_G$, where $\mathcal T_G$ is the canonical $P_S$-torsor over $X_G$.

The Zariski cohomology groups $H^i(X_G,\cV)$ and $H^i(X_G,\cV(-D))$ are independent, up to canonical isomorphism, of the choice of cone decomposition $\Sigma$ for the compactification, and have actions of prime-to-$p$ Hecke operators $[KgK]$, for $g \in G(\AA_{\fin}^p)$. The same is true for $H$ in place of $G$, and hence there are morphisms of sheaves
\begin{equation}\label{pullback}
H^i(X_G,\cV) \xrightarrow{\iota^*} H^i(X_H, \cV_{B_H})
\end{equation}
and also
\begin{equation}\label{pushforward}
H^i(X_H, \cV_{B_H} \otimes \alpha_{G/H}^{-1}) \xrightarrow{\iota_*} H^{i+1}(X_G,\cV)
\end{equation}
for $0 \leq i \leq 2$ and any $P_S$-representation $V$. Here, $\alpha_{G/H}$ denotes the character $(1, 1; 0)$ of $B_H$, the Borel subgroup of $H$. These maps will play a key role later for defining the appropriate pairings involved in the main constructions of the note.

Let $\pi$ (resp. $\sigma$) be a cuspidal automorphic representation of $G$ (resp. $\GL_2$), and consider the arithmetic normalisation of the finite part, $\pi_{\fin}' := \pi_{\fin} \otimes \| \cdot \|^{-(r_1+r_2)/2}$ (resp. $\sigma_{\fin}' := \sigma_{\fin} \otimes \| \cdot \|^{-c_1/2}$. Let $L_1$ stand for the irreducible $M_S$-representation with highest weight $L_1 \, : \, \lambda(r_1+3,1-r_2)$. Similarly, $L_2$ is the irreducible $M_S$-representation with highest weight $L_2 \, : \, \lambda(r_2+2,-r_1)$. The following result is a consequence of Arthur's classification:

\begin{theorem}
 Let $i=1,2$. If $\pi$ is of general type or of Yoshida type, then $\pi_{\fin}'$ appears with multiplicity one as a Jordan--H\"older factor of the $G(\AA_{\fin})$-representations
 \[ H^{3-i}(X_{G,\QQ},\cV_{\kappa_i(\nu)}(-D)) \otimes \CC \quad \text{ and } \quad H^{3-i}(X_{G,\QQ},\cV_{\kappa_i(\nu)}) \otimes \CC. \]
 Moreover, it appears as a direct summand of both representations, and the map between the two is an isomorphism on this summand.
 If $L$ is any irreducible representation of $M_S$ which is not isomorphic to $V_{\kappa_i(\nu)}$, then the localisations of $H^{3-i}(X_{G,\QQ},[L](-D))$ and $H^{3-i}(X_{G,\QQ},[L])$ at the maximal ideal of the spherical Hecke algebra associated to the $L$-packet of $\pi$ are zero for all $i$.\qed
\end{theorem}

We write $H^i(\pi_{\fin})$ for the $\pi_{\fin}'$-isotypical component of $H^i\left(X_{G,E},[L_1](-D)\right)$, for some number field $E$ over which $\pi_{\fin}'$ is definable.

\section{A pairing on coherent cohomology}

 In this section we define a pairing between coherent cohomology groups which we will later use to study $L$-values for region $(D^-)$.

 \subsection{Automorphic forms as coherent cohomology classes}

  We fix a weight $\nu = (r_1, r_2; c) = (k_1 - 3, k_2 - 3; c)$, for integers $k_1 \ge k_2 \ge 2$. Then there are two discrete-series representations of $\GSp_4(\RR)$ of infinitesimal character $\nu^\vee + \rho$: a holomorphic discrete series $\pi_\infty^H$, corresponding classically to holomorphic Siegel modular forms of weight $(k_1, k_2)$, which contributes to cohomology in degrees 0 and 3; and a generic discrete series $\pi_\infty^{W}$, which contributes in degrees 1 and 2.

  More canonically, we can write this as follows. Let $K_{\infty} = \RR^{\times} \cdot U_2(\RR)$ denote the maximal compact-mod-centre subgroup of $G(\RR)_+$. The representation $\pi_{\infty}^W$ has two direct summands as a $G(\RR)_+$-representation, $\pi_{\infty}^W = \pi_{\infty, 1} \oplus \pi_{\infty, 2}$, which have minimal $K_{\infty}$-types $\tau_1=(r_1+3, -r_2-1)$ and $\tau_2 = (r_2+1, -r_1-3)$, respectively. Since the minimal $K_{\infty}$-type in an irreducible discrete series has multiplicity 1, we have $\dim \Hom_{K_{\infty}}(\tau_i, \pi_{\infty}) = 1$ for $i=1,2$. Then, for each automorphic representation $\pi$ whose Archimedean component is $\pi_\infty^W$, we have a canonical isomorphism of irreducible smooth $G(\AA_{\fin})$-representations
  \[ \Hom_{K_{\infty}}(\tau_i, \pi) \{ \frac{r_1+r_2}{2} \} \cong H^{3-i}(X_{G, \QQ}, \cV_{\kappa_i}(-D))_{\CC}[\pi_\mathrm{f}]. \]

  \begin{remark}
   For weights $(k_1, k_2)$ sufficiently far from the walls of the Weyl chamber (so there are no non-tempered representations contributing to the cohomology), this is proved in \cite{harriskudla92}. It follows from the results of \cite{su-preprint}, together with Arthur's classification of discrete-series representations of $\GSp_4$, that the result in fact applies for all weights.
  \end{remark}

  Given $\xi \in H^{3-i}(X_{G, \QQ}, \cV_{\kappa_i}(-D))_{\CC}[\pi_\mathrm{f}]$, we denote by $F_\xi$ the corresponding homomorphism $\tau_i \to \pi$; we may consider $F_{\xi}$ as a harmonic vector-valued cusp form on $G$, taking values in the representation $\tau_i^\vee$.

  We use similar notations for $\GL_2$: for any $k \ge 1$, the space $H^0(X_{\GL_2}, \cV_{(-k)})$ is isomorphic to weight $k$ modular forms; and $H^1(X_{\GL_2}, \cV_{(k-2)})$ is isomorphic to the space of anti-holomorphic modular forms of weight $-k$ (i.e.~complex conjugates of holomorphic forms of weight $k$). We write $\omega \to F_{\omega}$ and $\eta \to F_{\eta}$ for these isomorphisms.

 \subsection{Whittaker periods}

  Suppose $\pi$ is a cuspidal automorphic representation with Archimedean component $\pi_\infty^W$, and which is globally generic. Let $E$ be the coefficient field of $\pi$.

  If we choose $j \in \{1, 2\}$ then there is a canonical basis of $\Hom_{K_\infty}(\tau_{3-j}, \cW(\pi_\infty^W))$ computed by Moriyama (which we shall recall in more detail below). From this we obtain two $E$-rational structures on $\Hom_{K_\infty}(\tau_j, \pi)$: one via the isomorphism to coherent $H^j$, and one via tensoring Moriyama's basis vector at $\infty$ with the canonical $E$-structure on the Whittaker model of $\pi_{\fin}$. Since $\pi_{\fin}$ is irreducible, these must differ by a constant in $\CC^\times / E^\times$.
  \begin{notation}
   We denote this scalar factor by $\Omega^{(j)}_{\pi} \in \CC^\times / E^\times$, the \emph{$H^j$ Whittaker period} of $\pi$.
  \end{notation}

  The period denoted by $\Omega^W_{\pi}$ in \cite{LPSZ1} (appearing in rationality results for case $(F)$) is the $H^2$ Whittaker period. In the present work, it is the $H^1$ Whittaker period which appears instead. (It is far from clear \emph{a priori} how these periods are related to each other.)

  Similar considerations apply to holomorphic cuspidal $\GL_2$ representations $\sigma$. In this case we obtain a period for $H^0$ and a period for $H^1$. We may choose our standard Whittaker functions at $\infty$ so that the Whittaker-rational classes are those whose $q$-expansions have coefficients in $E$. By the $q$-expansion principle the $H^0$ Whittaker period is just 1. On the other hand, since the Serre duality pairing on coherent cohomology preserves the $E$-rational structures, and this duality pairing corresponds to the Petersson product on automorphic forms, the $H^1$ Whittaker period must be given by the Petersson norm of the normalised newform generating $\sigma$.

 \subsection{Pullback to $H$}

  Now let $(c_1,c_2)$ be a pair of integers, with $c_1 + c_2 = k_1 + k_2 \bmod 2$, and satisfying the inequalities
  \[ 1 \le c_1, \quad k_1-k_2+2 \leq c_2-c_1, \quad c_2 \leq k_1 \]
  defining the region $(D^-)$. We denote by $\lambda$ the weight $(-c_1, c_2 - 2)$ for $H$. We want to define a pairing
  \begin{equation}
   \label{eq:pairing}
   H^1\left(X_{G, \QQ}, \cV^G_{\kappa_2}(-D)\right) \times H^1\left(X_{H, \QQ}, \cV^H_{\lambda}\right) \to \QQ.
  \end{equation}
  Let us define $t = \frac{(c_2 - c_1) - (k_1 - k_2 + 2)}{2} \geq 0$.

 \subsection{The $t = 0$ case}

  Let us first suppose that we have $c_2 - c_1 = k_1 - k_2 + 2$; this is the situation studied in \S 2.5 of \cite{harriskudla92}. We have $\kappa_2(\nu) = (k_2 - 4, -k_1)$ and one computes that the pullback of $\cV_{\kappa_2(\nu)}$ to $H$ is given by the direct sum
  \[ \bigoplus_{0 \le j \le k_1 + k_2 - 4} \cV^H_{(k_2 - 4 -j, j-k_1)}. \]
  Since $c_2 \le k_1$, this implies that one of the direct summands in $\iota^*(\cV_{\kappa_2(\nu)})$ is $\cV^H_{(c_1 - 2, -c_2)}$; so we obtain a pullback map
  \begin{align*}
   \iota^* : H^1(X_{G, \QQ}, \cV^G_{\kappa_2(\nu)}(-D)) &\to H^1(X_{H, \QQ}, \cV^H_{(c_1 - 2, -c_2)}(-D)) \\
    &= \left[  H^1(X_{H, \QQ}, \cV^H_{(-c_1, c_2-2)})\right]^\vee \quad\text{(by Serre duality)},
  \end{align*}
  defining a pairing between the groups \eqref{eq:pairing}, which can be explicitly expressed as the integral
  \[ \langle \xi, \omega \boxtimes \eta \rangle =
   \frac{1}{(2\pi i)^2}\int_{\RR^\times H(\QQ) \backslash H(\AA)} F_{\xi}(v_{c_1, -c_2})(\iota(h)) F_\omega(h_1) F_\eta(h_2)\, \mathrm{d}h, \]
  where $v_{c_1, -c_2}$ is the weight $(c_1, -c_2)$ standard basis vector of $\tau_2$. Note that the integrand is invariant under $K_\infty$. (It is also invariant under a finite-index subgroup of $\AA^\times$; in practice we will be interested in the case when $\xi$, $\omega$, and $\eta$ have product 1, so we may instead take the integral over $\AA^\times H(\QQ) \backslash H(\AA)$.)

 \subsection{General $t \ge 0$}

  In general, let us write $c_1' = c_1 + 2t = c_2 - (k_1 - k_2 + 2)$, so that $(c_1', c_2)$ satisfies the assumptions of the previous section. Then the \emph{Maass--Shimura derivative} $\delta^t$ sends the holomorphic form $F_\omega$ to a non-holomorphic automorphic form $F^{(t)}_\omega$ of $K_\infty$-type $-c_1'$ (i.e.~a nearly-holomorphic form of weight $c_1'$ in the sense of Shimura). So we may form the more general integral
  \[ \langle \xi, \delta^t(\omega) \boxtimes \eta\rangle =
   \frac{1}{(2\pi i)^2}\int_{\RR^\times H(\QQ) \backslash H(\AA)} F_{\xi}(v_{c_1', -c_2})(\iota(h))
   \delta^t F_\omega(h_1) F_\eta(h_2)\, \mathrm{d}h,
  \]
  which \emph{a priori} defines a pairing between the groups in \eqref{eq:pairing} after extension to $\CC$; our goal is to show that it respects the rational structures.

  As in many previous works (e.g. \cite{Urban-nearly-overconvergent, darmonrotger14}), we can interpret $F^{(t)}_\omega$ as a section of a larger vector bundle $\widetilde{\cV}_{(-c1')} \supseteq \cV_{(-c_1')}$, corresponding to a reducible representation of $B_{\GL_2}$. Arguing as in \cite{LPSZ1}, we can find a sheaf $\widetilde{\cV}^G_{\kappa_2} \twoheadrightarrow \cV^G_{\kappa_2}$ (defined as a subquotient of the Hodge filtration on the sheaf attached to the algebraic representation of weight $\nu$), and a pairing $\iota^*\left(\widetilde{\cV}^G_{\kappa_2}\right) \times  \widetilde{\cV}^H_{(-c1', c_2 - 2)} \to \Omega^1_{X_H}$, which is compatible with the obvious pairing $\iota^*\left(\cV^G_{\kappa_2}\right) \times \cV^H_{(-c1', c_2 - 2)} \to \Omega^1_{X_H}$. Moreover, the map on $H^1(X_{G, \QQ}, -)$ induced by the quotient map $\widetilde{\cV}^G_{\kappa_2} \twoheadrightarrow \cV^G_{\kappa_2}$ is an isomorphism on the $\pi_{\fin}$-eigenspace. So we can interpret $\langle \xi, \delta^t(\omega) \boxtimes \eta\rangle$ as a cup-product in the cohomology of these larger sheaves; in particular, it respects the $E$-structures on the cohomology groups, where $E$ is the rationality field of $\pi_{\fin}$.

  \begin{remark}
   If $(c_1, c_2)$ lies in region $(D)$, but not in the subregion $(D^-)$, then the above construction does not work, because the $K^H_{\infty}$-type $(c_1', -c_2)$ no longer appears in $\tau_2$. It seems possible that this can be ``repaired'' by applying differential operators to $F_{\xi}$, rather than to the two $\GL_2$ factors; we hope to investigate this further in a future work.
  \end{remark}

\section{Rationality results for Gan--Gross--Prasad periods}

  Let $\pi$ be as in the previous section; and let $\sigma_1, \sigma_2$ be cuspidal automorphic representations of $\GL_2$, generated by holomorphic cuspidal modular forms of weights $c_1$ and $c_2$ respectively, such that $(c_1, c_2)$ lies in region $(D^-)$ of our diagram. We suppose that $\omega$ and $\eta$ are in the $\sigma_1$, resp.~$\sigma_2$, isotypic part of the cohomology groups.

  \begin{theorem}\label{thm:GGP}
   If all three classes $\xi, \omega, \eta$ are defined over some number field $E$ as coherent cohomology classes, then the period $\langle \xi, \delta^t(\omega) \boxtimes \eta\rangle$ is in $E$.
  \end{theorem}

  Since the relation between rationality as coherent cohomology classes and rationality in the Whittaker model is given by the periods defined above, we can reformulate this as follows:

  \begin{corollary}
   If $\xi$, $\omega$ and $\eta$ are defined over $E$ in the respective Whittaker models, then we have
   \[ \frac{\langle \xi, \delta^t(\omega) \boxtimes \eta\rangle}{\Omega^{(1)}_\pi \cdot \langle h, h \rangle} \in E, \]
   where $h$ is the normalised holomorphic newform generating $\sigma_2$.
  \end{corollary}

  We briefly recall the relation between this period and central $L$-values. If $\pi$ and $\sigma_1 \otimes \sigma_2$ have trivial central characters (and thus factor through $\SO_5$ and $\SO_4$ respectively), and the local root numbers $\varepsilon_v(\pi_v \times \sigma_{1, v} \times \sigma_{2, v})$ are $+1$ for all finite places, then the Ichino--Ikeda conjecture \cite{ichinoikeda10} predicts a formula for the absolute value of the global period. This has the form
  \[ \frac{|\langle \xi, \delta^t(\omega) \boxtimes \eta\rangle|^2}{\|F_\xi\|^2 \cdot \|F_{\delta^t(\omega) \boxtimes \eta}\|^2} = (\star) \cdot \frac{L(\pi \times \sigma_1 \times \sigma_2, \tfrac{1}{2})}{L(\operatorname{ad} \pi, 1) L(\operatorname{ad} \sigma_1, 1) L(\operatorname{ad}\sigma_2, 1)} \prod_v c_v, \]
  where $(*)$ is an explicit factor, and $c_v$ are local matrix coefficients (which are nonzero, and equal to 1 for all but finitely many places). So if the Ichino--Ikeda conjecture holds, then Theorem \ref{thm:GGP} determines $L(\pi \times \sigma_1 \times \sigma_2, \tfrac{1}{2})$ up to an algebraic factor (although making this explicit would involve computing the local matrix coefficient $c_v$, which is a nontrivial task, particularly for $v = \infty$).

\section{Novodvorsky's zeta integral}

 We now return to the case considered in the introduction, so $\pi \times \sigma$ is an automorphic representation of $\GSp_4 \times \GL_2$ and the weights $(k_1, k_2, \ell)$ satisfy the inequalities $(D)$ of Table \ref{table1}. We consider Novodvorsky's zeta integral
 \[ Z(F_0, \Phi_1, F_2; s) = \int_{Z_H(\AA) H(\QQ) \backslash H(\AA)} F(h) E^{\Phi_1}(h_1; \chi, s) F'(h_2) \mathrm{d}h, \]
 where $F$ and $F'$ are automorphic forms in $\pi$ and $\sigma$ respectively, and $E^{\Phi_1}(h_1; \chi, s)$ is an Eisenstein series, depending on a choice of Schwartz function $\Phi_1 \in \mathcal{S}(\AA^2)$.

 \subsection{Expression via coherent cohomology}

  We first show how this zeta integral can be interpreted as a coherent cup product of the type considered in the previous sections, but with $\omega$ an Eisenstein, rather than cuspidal, form. We take $F_0 = F_\xi$ and $F_1 = F_\eta$, for coherent cohomology classes $\xi$, $\eta$ contributing to $H^1(X_G)$ and $H^1(X_{\GL_2})$, as before (so we are taking $c_2 = \ell$). We suppose $\ell \le k_1$ (so we are in case $(D^-)$).

  As explained in \cite{LPSZ1}, for suitable choices of parameters $E^{\Phi_1}(h_1; \chi, s)$ is a nearly-holomorphic Eisenstein series.

  We take $\Phi_{1, \infty}(x, y) = 2^{1-c_1'} (x + iy)^{c_1'} \exp(-\pi(x^2 + y^2))$, where $c_1' = \ell - (k_1 - k_2 + 2) \ge 1$; and let $\Phi$ be any Schwartz function of the form $\Phi_{1, \fin} \times \Phi_{1, \infty}$, for this particular $\Phi_{1, \infty}$ and any Schwartz function on $\AA_{\fin}^2$. The condition for $s$ to be a critical value is precisely that $s = \tfrac{c_1'}{2} \bmod \ZZ$ and $|2s - 1| \le c_1' - 1$; and for $s$ satisfying this, $E^{\Phi_1}(-; \chi, s)$ is nearly-holomorphic (of weight $c_1'$. Moreover, if we let $c_1 = 1 + |2s - 1|$ and $t = \tfrac{c_1' - c_1}{2} \in \ZZ_{\ge 0}$, then there is a holomorphic Eisenstein series of weight $c_1$ whose image under $\delta^t$ is $E^{\Phi_1}(-; \chi, s)$. (For $\Phi_{1, \fin}$ of a certain specific form this is proved in \cite[Corollary 5.2.1]{leiloefflerzerbes14}; the general proof is no different.)

  \begin{remark}
   One can check that  $\Phi_{1, \fin}$ takes values in our fixed number field $E$, then $E^{\Phi_1}(-; \chi_1, s)$ is defined over that number field (as a coherent cohomology class).
  \end{remark}

  \begin{corollary}
   If $\Phi_{\fin}$ is $E$-valued, and $\xi$, $\eta$ are defined over $E$ as coherent cohomology classes, then $Z(F_0, \Phi_1, F_2; s)$ lies in $(2\pi i)^2 E$.
  \end{corollary}

 \subsection{Eulerian factorisation}

  As explained in \cite{LPSZ1}, if $\pi$ and $\sigma$ are generic, then we can write this as an integral in terms of the Whittaker functions $W_0$ and $W_2$ associated to $F_0$ and $F_2$. If the data $W_0, \Phi_1, W_2$ are factorisable as products of local data, then the global integral has a corresponding factorisation as $\prod_v Z_v(W_{0, v}, \Phi_{1, v}, F_{2, v}; s)$, where
  \[ Z_v(W_{0, v}, \Phi_{1, v}, F_{2, v}; s) =
   \int_{(Z_H N_H\backslash H)(\QQ_v)} W_{0, v}(h) f^{\Phi_{1, v}}(h_1; \chi_v, s) W_{2, v}(h_2)\, \mathrm{d}h. \]
  For all but finitely many places, the local integral $Z_v$ is equal to the $L$-factor $L(\pi_v \times \sigma_v, s)$, so we can write
  \[ Z(F_0, \Phi_1, F_2; s) = L(\pi \times \sigma, s) \cdot Z_\infty(W_{0, \infty}, \Phi_{1, \infty}, W_{2, \infty}; s) \cdot  \prod_{\ell \in S} \frac{Z_v(W_{0,\ell}, \Phi_{1, \ell}, W_{2, \ell}; s)}{L(\pi_\ell \times \sigma_\ell, s)}, \]
  where $S$ is a finite set of (finite) primes.

  For $v = \ell$ a finite prime in $S$, the ratio $\dfrac{Z_\ell(W_{0, \ell}, \Phi_{1, \ell}, W_{2, \ell}; s)}{L(\pi_\ell \times \sigma_\ell, s)}$ is a rational function in $\ell^{-s}$ and $\ell^s$. The local assumption on $\pi \times \sigma$ in \cref{thm:main} implies that this function is actually a polynomial, and the ideal generated by these polynomials (as $W_{0, \ell}, \Phi_{1, \ell}, W_{2, \ell}$ vary) is the unit ideal. One can check that if $s = \tfrac{w}{2} \bmod \ZZ$, and the data $(W_{0, \ell}, \Phi_{1, \ell}, W_{2, \ell})$ are defined over $E$, then the zeta integral is itself $E$-valued.

  If we choose test data of this form at the local places, and equal to the standard Moriyama test data at $\infty$, then we can conclude that
  \[ Z(F_0, \Phi_1, F_2; s) \in E^\times \cdot Z_\infty \cdot L(\pi \times \sigma, s)\]
  where $Z_\infty$ is the local zeta integral for Moriyama's standard test data at $\infty$; but we also have
  \[ Z(F_0, \Phi_1, F_2; s) \in E^\times \cdot (2\pi i)^2 \langle h ,h \rangle \Omega^{(1)}_{\pi}.\]

  Thus, in order to prove \cref{thm:main}, it remains to compute the Archimedean local integral $Z_\infty$.

\section{Archimedean zeta integrals}

\subsection{Zeta-integral generalities}

 For $\Pi$ a smooth representation of $\GSp_4(F)$, where $F$ is a local field (archimedean or not), we have the two-parameter $\GSp_4 \times \GL_2$ zeta-integral \[ Z(W, \Phi_1, \Phi_2; \chi_1, \chi_2, s_1, s_2). \]

 If we write this in terms of Bessel models, it is
 \begin{equation}\label{eq:Bessel}
 \int_{D N_H \backslash H} B_W(h; s_1 - s_2 + \tfrac{1}{2}) f^{\Phi_1}(h_1; s_1, \chi_1) f^{\Phi_2}(h_2; s_2, \chi_2) \mathrm{d}h.
 \end{equation}

 We have the Iwasawa decomposition $H = B_H K_H$, where $K_H$ is the maximal compact. If the data are chosen so that the integrand is $K$-invariant, then (using the fact that the $f^{\Phi}$'s live in principal-series representations, so have a known transformation property under $B_H$) we obtain
 \[  Z(W, \Phi_1, \Phi_2, s_1, s_2) = f^{\Phi_1}(1; \chi_1, s_1) f^{\Phi_2}(1; \chi_2, s_2)
    \int_{F^\times} B_W(\begin{smatrix} t \\ & t \\ &&1 \\ &&&1\end{smatrix}; s_1 - s_2 + \tfrac{1}{2}) |t|^{(s_1 + s_2 - 2)} \, \mathrm{d}^\times t.
 \]

 If $F = \RR$, then we have an explicit formula for $\int_{F^\times} B_W(\dots)(\dots)$ due to Moriyama, which we recall below.

 We want to use this to study $\GSp_4 \times \GL_2$ zeta-integrals. For $c \in \ZZ_{\ge 1}$ there is a (limit-of-)discrete-series representation $D_c^+$ of $\SL_2(\RR)$, corresponding to weight $c$ holomorphic modular forms, whose lowest $K$-type is $\stbt {\cos \theta}{\sin \theta}{-\sin\theta}{\cos\theta} \mapsto e^{i c \theta}$. The Whittaker function of the lowest $K$-type vector is given along the torus by \[ W^{(k)}(\stbt t 0 0 1) = t^{c/2} e^{-2 \pi t} \] (up to an arbitrary scalar, but if we want our Whittaker functions to match up with the conventional notion of $q$-expansions then this is clearly the good normalisation).

 We can embed $D_c^+$ into a principal-series representation $I(|\cdot|^{s-1/2}, |\cdot|^{1/2 - s} \chi^{-1})$, taking $\chi = (\text{sign})^c$ and $s = \tfrac{c}{2}$. Then the principal series is reducible, with $D_c^+ \oplus D_c^-$ as a subrepresentation and a finite-dimensional representation as quotient.

 If we consider the function $\Phi^{(c)}(x, y) = 2^{1-c} (x + iy)^c e^{-\pi(x^2 + y^2)}$, then $f^{\Phi^{(c)}}(-; \operatorname{sign}^c, s)$ has the same $K$-type, and it lives in $I(|\cdot|^{s-1/2}, |\cdot|^{1/2 - s} \operatorname{sign}^c)$. Specializing at $s = \tfrac{c}{2}$, $f^{\Phi^{(c)}}(-; \operatorname{sgn}^c, \tfrac{c}{2})$ must land in the $D_c^+$ subrepresentation, since its $K$-type does not appear in any of the other factors. Hence, its image under the Whittaker transform, $W^{\Phi^{(c)}}(-, \operatorname{sgn}^c, c/2)$, must be a scalar multiple of the above Whittaker function.

 If we evaluate the Whittaker function $W^{\Phi^{(c)}}(-, \operatorname{sgn}^c, c/2)$ at the identity, we are led to a rather nasty definite integral, which eventually turns out to be $e^{-2\pi}$. This is the value at 1 of the normalised Whittaker function above, so our normalisations are compatible (the archimedean analogue of the compatibility noted in \cite[p4097]{LPSZ1}). That is, if we substitute $\Phi_2 = \Phi^{(c_2)}$, $s_2 = \tfrac{c_2}{2}$, and $\chi_2 = \operatorname{sgn}^{c_2}$ in Moriyama's formulae, and let $s = s_1$, we obtain a formula for the $\GSp_4 \times \GL_2$ zeta integral $Z(W, \Phi_1, W^{(c_2)}; s)$.

 \subsection{Choosing the parameters}

  Let $(r_1, r_2)$ be integers with $r_1 \ge r_2 \ge -1$. This determines an $L$-packet of representations of $\GSp_4(\RR)$, which are discrete-series if $r_2 \ge 0$ and limit-of-discrete-series if $r_2 = -1$, as usual.

  \begin{notation}
  In the notations of Moriyama's paper \cite{moriyama04}, let $(\lambda_1, \lambda_2) = (r_1 + 3, -1-r_2)$; and set $d = \lambda_1 - \lambda_2 = r_1 + r_2 + 4$.
  \end{notation}
  Then we have the inequalities $1 - \lambda_1 \le \lambda_2 \le 0$ that Moriyama requires (and in fact strict inequality holds). Attached to $(\lambda_1, \lambda_2)$, Moriyama defines a pair of discrete / limit-of-discrete series $\Sp_4(\RR)$-representations $D_{(\lambda_1, \lambda_2)}$ and $D_{(-\lambda_2, -\lambda_1)}$, with $D_{(\lambda_1, \lambda_2)}$ contributing to coherent $H^1$, and $D_{(-\lambda_2, -\lambda_1)}$ to coherent $H^2$. Note that the Whittaker functions of $D_{(-\lambda_2, -\lambda_1)}$ are supported on $\GSp_4^+(\RR)$, while the Whittaker functions of the dual representation are supported on the non-identity component.

  We let $\Pi_\infty$ be the unique representation of $\GSp_4(\RR)$ whose restriction to $\Sp_4(\RR)$ is $D_{\lambda_1, \lambda_2} \oplus D_{-\lambda_2, -\lambda_1}$, and whose central character is trivial on $\RR_{> 0}$.

  \begin{remark}
   Our parametrisation of the $K$-types follows \cite{harriskudla92}, and unfortunately the conventions of Harris--Kudla and Moriyama are not the same; so in our notations, $D_{(\lambda_1, \lambda_2)}$ has minimal $K$-type $(-\lambda_2, -\lambda_1)$ and $D_{(-\lambda_2, -\lambda_1)}$ has highest $K$-type $(\lambda_1, \lambda_2)$ (sic!).
  \end{remark}

\subsection{Moriyama's result}

Moriyama defines explicit Whittaker functions $W_k$, for $0 \le k \le \lambda_1 -\lambda_2$, giving a basis of the minimal $K$-type of $D_{(-\lambda_2, -\lambda_1)}$. In Proposition 8 of \emph{op.cit.} he states a formula for the integral
\[Z(s, y_1; W) =  \int_{y \in \RR^\times} \int_{x \in \RR} W\left(
\begin{pmatrix}
 uy \\ &  y \\ & x & 1 & \\ &&& u^{-1}
\end{pmatrix}\right) |y|^{s - 3/2}\, \mathrm{d}x\, \mathrm{d}^\times y,
\]
for $u > 0$ an auxiliary parameter. This involves a Mellin inversion integral, and we get rid of this by taking the forward Mellin integral to get a formula for the Mellin transform of the Bessel function along the torus. (NB: Moriyama uses the other model of $\GSp_4$, with the last two rows \& last two columns of the matrix switched.) Unravelling this, we get the following formula: for $0 \le k \le d$, if $W_k$ is the vector of $K_H^\circ$-type $(-r_1 - 3 + k, r_2 + 1 - k)$, then\footnote{Note there is a typo on p.~4108 of \cite{LPSZ1}, we erroneously flipped the two components of the $K_H^\circ$-type. Accordingly, the value of $k$ given there is wrong and should be replaced with $d - k$.}
\[ \int_{\RR^\times} B_{W_k}(\begin{smatrix} u \\ & u \\ &&1 \\ &&&1 \end{smatrix}; s_1 - s_2 + \tfrac{1}{2}) |u|^{s_1 + s_2 - 2}\, \mathrm{d}u =
C \cdot \frac{(-1)^k L(\Pi_\infty, s_1 - s_2 + \tfrac{1}{2}) L(\Pi_\infty, s_1 + s_2 - \tfrac{1}{2})}
{\pi^{s_1 + s_2 - \tfrac{1}{2}} \Gamma(s_1 + \tfrac{r_1 + 3 - k}{2}) \Gamma(s_2 + \tfrac{-1-r_2 + k}{2})},\]
where $C$ is some constant (depending on $(r_1, r_2)$ but not on any of the other data). This is \cite[Theorem 8.21]{LPSZ1}.

\subsection{Region $(F)$ case revisited}

For region $(F)$, we apply this to compute $Z(W_k, \Phi^{(c_1)}, W^{(c_2)}; s)$ for integers $(c_1, c_2)$ with $c_i \ge 1$ and $c_1 + c_2 = r_1 - r_2 + 2$. We take $k = r_1 + 3 - c_1 = r_2 + 1 - c_2$. Then the $K$-type of $W_k$ is $(-c_1, -c_2)$, meaning it can pair nontrivially with a pair of holomorphic forms of weights $c_1$ and $c_2$.

Then we get
\[ C \cdot \frac{(-1)^k L(\Pi_\infty, s_1 - s_2 + \tfrac{1}{2}) L(\Pi_\infty, s_1 + s_2 - \tfrac{1}{2})}
{\pi^{s_1 + s_2 - \tfrac{1}{2}} \Gamma(s_1 + \tfrac{c_1}{2}) \Gamma(s_2 + \tfrac{c_2}{2})} \cdot f^{\Phi^{(c_1)}}(1, s_1) f^{\Phi^{(c_2)}}(1, s_2) \Bigg|_{(s_1, s_2) = (s, \tfrac{c_2}{2})}.\]
An explicit computation gives
\[ f^{\Phi^{(c)}}(1, s) = 2^{1-c} i^c \pi^{-(s + c/2)} \Gamma(s + \tfrac{c}{2}), \] so up to factors which don't depend on the $c_i$ and hence can be absorbed into $C$.

Moreover, for region $(F)$ we have
\[ L(\Pi_\infty, s_1 - s_2 + \tfrac{1}{2}) L(\Pi_\infty, s_1 + s_2 - \tfrac{1}{2}) = L(\Pi_\infty \times \Sigma_\infty, s). \] So we get $(\text{const}) \cdot (-1)^{c_2} \cdot L(\Pi_\infty \times \Sigma_\infty, s)$, which is (by definition of critical values) non-zero in the critical range.

\subsection{Region $(D)$}

Now we are going to take $c_1, c_2$ with $c_1 \ge 1$ and $c_2 - c_1 = r_1 - r_2 + 2$; and we choose
\[ k = r_1 + 3 + c_1 = c_2 + r_2 + 1. \]
The constraint $k \le d = r_1 + r_2 + 4$ corresponds to $c_2 \le r_1 + 3$, which is the inequality required for the ``bottom half'' of region $(D)$. Our test data will be
\[ Z(W_k, \stbt{-1}{}{}{1} \Phi^{(c_1)}, W^{(c_2)}), \]
and since acting by $\stbt{-1}{}{}{1}$ does not change the values of $f^{\Phi}$ along the torus, we can write this as
\[ C \cdot \frac{(-1)^k L(\Pi_\infty, s_1 - s_2 + \tfrac{1}{2}) L(\Pi_\infty, s_1 + s_2 - \tfrac{1}{2})}
{\pi^{s_1 + s_2 - \tfrac{1}{2}} \Gamma(s_1 - \tfrac{c_1}{2}) \Gamma(s_2 + \tfrac{c_2}{2})} \cdot f^{\Phi^{(c_1)}}(1, s_1) f^{\Phi^{(c_2)}}(1, s_2) \Bigg|_{(s_1, s_2) = (s, \tfrac{c_2}{2})}.\]
Note the change from $s_1 + \tfrac{c_1}{2}$ to $s_1 - \tfrac{c_1}{2}$ in the denominator.

On the other hand, in this case the numerator is not $L(\Pi_\infty \times \Sigma_\infty, s)$ any more; one computes that
\[ L(\Pi_\infty, s + c_2 - \tfrac{1}{2}) L(\Pi_\infty, s - c_2 + \tfrac{1}{2}) = L(\Pi_\infty \times \Sigma_\infty, s) \cdot (2\pi)^{c_1} \frac{\Gamma(s - \tfrac{c_1}{2})}{\Gamma(s + \tfrac{c_1}{2})}.\]
So the zeta-integral computes to
\[ \text{(const)} \cdot (-2\pi)^{c_2} \cdot L(\Pi_\infty \times \Sigma_\infty, s), \]
for some constant depending only on $(r_1, r_2)$; and we may choose our normalisation of the Archimedean Whittaker function so that this constant is 1.

\section{Local zeta integrals at $p$}

\subsection{Nonarchimedean $L$-factors}

We can extend the previous computations to non-archimedean primes. For that purpose, let $F$ be a non-archimedean field, and write $L(\pi \times \sigma,s)$ for the local $L$-factor associated to $\pi \otimes \sigma$ via Shahidi's method, as in \cite[\S4]{genestiertilouine05}. We recall the following result from \cite[Thm. 8.9]{LPSZ1}.

\begin{proposition}
The vector space of functions on $\CC$ spanned by the $Z(W,\Phi_1,\Phi_2; \chi_1, \chi_2, s_1,s_2)$, as the data $(W,\Phi_1,\Phi_2)$ vary, is a fractional ideal of $\CC[q^{\pm s}]$ containing the constant functions. If at least one of $\pi$ and $\sigma$ is unramified, this fractional ideal is generated by the $L$-factor $L(\pi \otimes \sigma, s)$. Further, there exists a canonical normalization for which the zeta integral is precisely $L(\pi \otimes \sigma, s)$.
\end{proposition}

The same considerations presented in \cite[\S8.3.1]{LPSZ1} regarding the rationality of the Whittaker models also apply.

\subsection{A special $H$-orbit on $\FL_G$}

From now on, let $p$ be a fixed prime. In this section, we discuss how to compute the local integrals at the prime $p$ for our choice of local conditions, and how this recovers the expected Euler factor.

\begin{lemma}
Let $\tau = \begin{smatrix} 1 \\ 1 & 1 \\ &&1\\&&-1&1\end{smatrix} \in M_{\Si}$, and let $\hat\tau = \tau w_1$. Then $H \hat\tau P$ is open in $G$, and $H \cap \hat\tau P \hat\tau^{-1}$ is a copy of $\GL_2$, embedded in $H$ via
\[\left(\tbt{a}{b}{c}{d}, \tbt{a}{-b}{-c}{d}\right),\]
and in $G$ via
\[ \left(\tbt{a}{b}{c}{d}, \tbt{a}{-b}{-c}{d}\right) \mapsto \begin{pmatrix} a & b & & b \\ c & d & c  \\ &&a&-b\\&&-c &d\end{pmatrix}.\]
\end{lemma}

\begin{proof} Elementary computation.
\end{proof}


\begin{lemma}
Suppose $\bar{n} \in \overline{N}(\Zp)$ is congruent to 1 modulo $p^k$, for $k \ge 1$. Then  we may write $\hat\tau \bar{n} = h \hat\tau p$ for some $h \in H(\Zp)$ and $p \in P(\Zp)$, with both $h$ and $p$ congruent to 1 modulo $p^k$.
\end{lemma}

\begin{proof}
This follows from the matrix identity
\[ h(x, y, z) \hat\tau \begin{smatrix} 1 \\ &1 \\ x & y & 1 \\ z & x & & 1 \end{smatrix} = \hat\tau p(x, y, z), \]
where
\[ h(x, y, z) = \left(\tbt{ (x+1) - \tfrac{yz}{x+1}} {\tfrac{y}{x+1} }{-z}{1} \tbt{1}{}{}{x+1}\right), \qquad
   p(x, y, z) = \begin{smatrix}
   x+1 & y & 0 & y/(x+1) \\
   0 & x+1 & 0 & 0 \\
   & & 1 & -y/(x+1)\\
   & & & 1\end{smatrix}.\qedhere\]
\end{proof}

\subsection{Siegel Jacquet modules}

Now let $\pi$ be an irreducible smooth generic representation of $\GSp_4(\Qp)$. We suppose $\pi$ is is the normalised induction of a representation $\rho \times \lambda$ of $M(\Qp)$. Note that we have
\[ L(\pi, s) = L(\lambda, s) L(\rho \otimes \lambda, s) L(\omega_{\rho} \lambda, s), \]
and the modulus character $\delta_P$ is $(A, c) \mapsto |\det(A) / c|^3$.

\begin{lemma}
The normalised Jacquet module with respect to $\overline{P}$, $J_{\overline{P}}(\pi) = \pi_{\overline{N}} \otimes \delta_{P}^{1/2}$, contains a unique subrepresentation isomorphic to $\rho \times \lambda$.
\end{lemma}

\begin{proof}
Bernstein's second adjointness theorem shows that $\Hom_M(\rho \times \lambda, J_{\overline{P}}(\pi)) = \Hom_G(\pi, \pi) \cong \CC$.
\end{proof}

Thus the unnormalised Jacquet module $\pi_{\overline{N}} = \pi / \pi(\overline{N})$ contains a canonical $M$-subrepresentation isomorphic to $(\rho \times \lambda) \otimes \delta_P^{-1/2}$. We write $\pi[\lambda]$ for the preimage in $\pi$ of this subrepresentation of $\pi_{\overline{N}}$. For any $v \in \pi[\lambda]$, we have $\diag(1, 1, x, x) v = |x|^{3/2} \lambda(x) v \bmod \pi(\overline{N})$.

\begin{note}
If $\rho$ and $\lambda$ are unramified, then a vector invariant under the depth $t$ Siegel parahoric $K_{G, \Si}(p^t)$ has this property if and only if it lies in the $U_1' = \alpha$ eigenspace (modulo the zero generalised eigenspace), where $\alpha = p^{3/2} \lambda(p)$ and $U_1'$ is the Hecke operator given by the double coset of $\diag(1, 1, p, p)$.
\end{note}

\subsection{Trilinear forms and Siegel Jacquet modules}

Let $\pi$ be as above, and let $\sigma_1, \sigma_2$ be irreducible $G$-representations with $\omega_\pi \omega_{\sigma_1} \omega_{\sigma_2} = 1$. One knows that $\Hom_{H(\Qp)}(\pi \times \sigma_1 \times \sigma_2, \CC)$ has dimension $\le 1$. We suppose it is nonzero, and choose a basis vector $\fz$.

\begin{remark}
If one or more of the $\sigma_i$ is principal-series, we can construct $\fz$ using Novodvorsky's $\GSp_4 \times \GL_2$ zeta integral.
\end{remark}
  %
  %

\begin{proposition}
For $x \in \pi$, $y_i \in \sigma_i$, we consider the sequence of elements $(z_k(x, y_1, y_2))_{k \ge 0}$ defined by
\[ z_k(x, y_1, y_2) = \left(\frac{\lambda(p)}{p^{3/2}}\right)^{-k} \fz\Big( \hat\tau s_k x, y_1, \stbt{-1}{}{}{1} y_2 \Big),\qquad s_k = \diag(1,1,p^k,p^k).\]
If $x \in \pi[\lambda]$, then $(z_k)$ is eventually constant, and its limiting value depends only on the image of $x$ in $\pi[\lambda] / \pi_{\overline{N}}$.
\end{proposition}

\begin{proof}
We first show that if $n \in \overline{N}(\Qp)$, then we have
\[ z_k(\overline{n}x, y_1, y_2) =  z_k(x, y_1, y_2) \qquad \forall k \gg 0\] (for any fixed $x, y_1, y_2$). Since the elements $s_k\overline{n} s_k^{-1}$ approach the identity as $k \to \infty$, for all $k \gg 0$ we can write \[ \hat\tau s_k \overline{n} s_k^{-1} = h_k \hat\tau p_k, \qquad h_k \in H(\Zp), p_k \in P(\Zp), \]
with $h_k$ and $p_k$ tending to 1 as $k \to \infty$. For all sufficiently large $k$, $h_k$ will fix $y_1 \otimes \stbt{-1}{}{}{1} y_2$, so we have
\[ z_k(\overline{n}x, y_1, y_2) = \left(\frac{\lambda(p)}{p^{3/2}}\right)^{-k} \fz\Big( \hat\tau s_k \gamma_k x, y_1, \stbt{-1}{}{}{1} y_2\Big), \qquad \gamma_k = s_k^{-1} p_k s_k. \]
Since conjugation by $s_k$ acts trivially on $M(\Qp)$, and shrinks $N(\Qp)$, the fact that $p_k \to 1$ certainly implies $\gamma_k \to 1$; so $\gamma_k x = x$ for sufficiently large $k$. This proves the claim.

Since $z_{k+1}(x, y_1, y_2) = \frac{p^{3/2}}{\lambda(p)} z_k(s_1x, y_1, y_2)$ by definition, and for $x \in \pi[\lambda]$ we have $s_1 x = \frac{\lambda(p)}{p^{3/2}} x \bmod \pi_{\overline{N}}$, it follows that $z_k$ is eventually constant for $x \in \pi[\lambda]$.
\end{proof}

\begin{definition}
For $\fz \in \Hom(\pi \times \sigma_1 \times \sigma_2, \CC)$, we write $\partial_{\Si}(\fz)$ for the trilinear form on $\pi[\lambda] \times \sigma_1 \times \sigma_2$ mapping $(x, y_1, y_2)$ to $\lim_{k \to \infty} z_k(x, y_1, y_2)$.
\end{definition}

One checks that for $A = \stbt{a}{b}{c}{d} \in \GL_2$, one has
\[
\partial_{\Si}(\fz)\left( \tbt{A}{}{}{\det(A)A'} x, A y_1, A y_2\right) = \partial_{\Si}(x, y_1, y_2).
\]
So we have defined a map
\begin{equation}\label{eq:homsetmap}
\partial_{\Si}: \Hom_H(\pi \times \sigma_1 \times \sigma_2, \CC) \to \Hom_{\GL_2}(\sigma_0 \times \sigma_1 \times \sigma_2, \CC),
\end{equation}
where $\sigma_0 = \rho \otimes \lambda$ is the restriction of $(\rho \times \lambda) \otimes \delta_P^{-1/2}$ to $\GL_2$, embedded in $M$ via $A \mapsto (A, \det (A))$. Note that both source and target of this map have dimension $\le 1$.

\begin{proposition}
\label{prop:stabilises1}
Suppose $x, y_1, y_2$ are invariant under the principal congruence subgroup modulo $p^t$ for some $t \ge 1$, and we have $U'_1 x = \alpha x$, where $\alpha = p^{3/2} \lambda(p)$. Then $x \in \pi[\lambda]$, and we have $\fz_k(x, y_1, y_2) = \fz_0(x, y_1, y_2)$ for all $k \ge 0$.
\end{proposition}

\begin{proof}
The relation $U'_1 x = \alpha x$ translates into $p^3 s_1 x = \alpha x \bmod \pi(\overline{N})$, since $U_1'$ is the composite of $s_1$ and a sum over $p^3$ coset representatives all lying in $\overline{N}$. Thus $x \in \pi[\lambda]$.

A similar argument shows that $\alpha^k \fz_0(x, y_1, y_2) = \fz_0( (U'_1)^k x, y_1, y_2) = p^{3k} \fz_0(s_k x, y_1, y_2)$, so $z_k(x, y_1, y_2) = z_0(x, y_1, y_2)$.
\end{proof}

\subsection{Relating the zeta-integrals}

We now identify $\pi$, $\sigma_1$, and $\sigma_2$ with their Whittaker models. More precisely, as in \cite{LPSZ1} we take Whittaker models for $\pi$ with respect to some additive character $\Psi$, and for $\sigma_1$ and $\sigma_2$ with respect to $\Psi^{-1}$.

If $\pi$, the $\sigma_i$, and $\Psi$ are all unramified, then there is a canonical spherical vector in the Whittaker model of each, and a unique $H$-invariant trilinear form $\fz^{\sph}$ satisfying $\fz^{\sph}(W_0^{\sph}, W_1^{\sph}, W_2^{\sph}) = 1$. Similarly, there is a spherical $\GL_2$-invariant trilinear form $\mathfrak{y}^{\sph}$ on $\sigma_0 \times \sigma_1 \times \sigma_2$. We identify $J_{\overline{P}}(W(\pi))$ with $W(\sigma_0)$ via mapping the normalised $U'_1$-eigenvector $W^{\prime, \Si}_{\alpha}$ to the normalised spherical vector.

\begin{proposition}\label{prop:relateintegralsI}
In this unramified setting we have
\[ \partial_{\Si}(\fz^{\sph}) = \Delta \cdot \mathfrak{y}^{\sph}, \qquad \Delta \coloneqq \frac{p^2}{(p^2-1)} \cdot L(\omega_\rho \lambda \times \sigma_1 \times \sigma_2, \tfrac{1}{2})^{-1}. \]
\end{proposition}

Note that $\omega_\pi = \lambda^2 \omega_\rho$, so
\[ L(\omega_\rho \lambda \times \sigma_1 \times \sigma_2, s) = L(\lambda^\vee \times \sigma_1^\vee \times \sigma_2^\vee, \tfrac{1}{2}).\]

\begin{proof}
Since $\mathfrak{y}^{\sph}$ is 1 at the spherical data, and the spherical vector of $\sigma_0$ is the image of $W'_{\alpha}$, it suffices to calculate
\[ \partial_{\Si}(\fz^{\sph})\left(W^{\prime, \Si}_{\alpha}, W_1^{\sph}, W_2^{\sph}\right)= \fz^{\sph}\left(\hat\tau W^{\prime, \Si}_{\alpha}, W_1^{\sph}, \stbt{-1}{}{}{1} W_2^{\sph}\right) = \fz^{\sph}\left(\hat\tau W^{\prime, \Si}_{\alpha}, W_1^{\sph}, W_2^{\sph}\right).\]
Since $\hat\tau$ lies in the same $H(\Zp)$-orbit on $G / P_{\Si}$ as the element $\eta J$ appearing in \cite{LZ20-zeta2}, we can apply the formulae of \emph{op.cit.} for Iwahori-level Shintani functions to obtain the stated result. In the notation of \emph{op.cit.} for the Hecke parameters, we have
\[ L(\lambda^\vee \times \sigma_1^\vee \times \sigma_2^\vee, \tfrac{1}{2}) =
\left(1 - \frac{p^2}{\alpha \fa_1\fa_2}\right)
\left(1 - \frac{p^2}{\alpha \fa_1\fb_2}\right)
\left(1 - \frac{p^2}{\alpha \fb_1\fa_2}\right)
\left(1 - \frac{p^2}{\alpha \fb_1\fb_2}\right).\qedhere \]
\end{proof}

 \subsection{Expansion along $\sigma_2$} We now 
  perform a second ``reduction along a Jacquet module'' argument.

  \begin{proposition}\label{prop:relateintegrals2}
   Let $W_0, W_1 \in \cW(\sigma_0), \cW(\sigma_1)$, and choose $t \ge 1$ such that $W_0, W_1$ are fixed by $\stbt{1}{}{p^t \Zp}{1}$. Then the value
   \[ \mathfrak{y}^{\sph}\left(W_0, W_1, W'_{\mathfrak{a}_2}[\ell] \right)\]
   is independent of $\ell \ge t$. Its limiting value is equal to the value at $s = \tfrac{1}{2}$ of the function defined for $\Re(s) \gg 0$ by
   \[ \Delta'_s \cdot \int_{\Qp^\times} W_0\left( \stbt x 0 0 1 \right)W_1\left(\stbt x 0 0 1 \right) \tau(x)|x|^{s-1}\, \mathrm{d}x,\]
   which has analytic continuation as a polynomial in $p^{\pm s}$; here $\tau$ is the unramified character sending $x$ to $p^{-1/2} \mathfrak{b}_2$, and
   \[ \Delta_s' =
    \frac{p}{(p+1)} L(\sigma_0 \times \sigma_1 \times \tau, s)^{-1} = \frac{p}{(p+1)}
    \left(1 - \frac{p^{2-s}}{\beta \mathfrak{a}_1 \mathfrak{a}_2}\right)
    \left(1 - \frac{p^{2-s}}{\beta \mathfrak{b}_1 \mathfrak{a}_2}\right)
    \left(1 - \frac{p^{2-s}}{\gamma \mathfrak{a}_1 \mathfrak{a}_2}\right)
    \left(1 - \frac{p^{2-s}}{\gamma \mathfrak{b}_1 \mathfrak{a}_2}\right).\]
  \end{proposition}

  \begin{proof}
   One can write down an explicit formula for $\mathfrak{y}^{\sph}$ as the leading term of the $\GL_2 \times \GL_2$ Rankin--Selberg zeta-integral (corresponding to deforming $\fa_2$ and $\fb_2$ to $\fa_2 |\cdot|^{1/2 - s}$ and $\fb_2 |\cdot|^{s - 1/2}$):
   \[ \mathfrak{y}^{\sph}\left(W_0, W_1, W_2 \right) = \lim_{s \to \tfrac{1}{2}} \frac{\langle f_2(g; s), R(g; s)\rangle}{L(\sigma_0 \times \sigma_1 \times \tau, s)}, \]
   where $f_2$ denotes the Siegel section corresponding to $W_2$, $\langle -, - \rangle$ denotes integration over $B_G \backslash G \cong \mathbf{P}^1$, and
   \[ R(g; s) \coloneqq \int_{\Qp^\times} W_0\left(\stbt x {}{}{1} g\right)W_1\left(\stbt x {}{} 1 g\right) \tau(x)|x|^{s-1}\, \mathrm{d}^\times x.\]
   For the particular choice of $W_2$ above, $f_2$ takes the value $p^\ell$ on the preimage of the identity in $\mathbf{P}^1(\ZZ / p^\ell)$, and its value is 0 elsewhere. Since the measure of this neighbourhood is $\tfrac{1}{p^{\ell-1}(p+1)}$, the integral over $\mathbf{P}^1$ is simply $\tfrac{p}{p+1} R(1)$.
  \end{proof}

 \subsection{Conclusions}

  Note that the product of the two $L$-factors $\Delta \cdot \Delta'_s|_{s = 1/2}$ appearing in \cref{prop:relateintegralsI} and \cref{prop:relateintegrals2} is a degree 8 factor of the degree 16 $L$-factor $L(\pi \times \sigma_1 \times \sigma_2, \tfrac{1}{2})$, and thus corresponds to an 8-dimensional direct summand of the 16-dimensional Weil--Deligne representation associated to $\pi \times \sigma_1 \times \sigma_2$; and the 8-dimensional subrepresentation which we obtain is precisely the one giving the ``Panchishkin subrepresentation'' of the Galois representation when the weights lie in region $(D)$. So the $L$-factor is the one denoted $\mathcal{E}^{(d)}$ in \cite{LZvista}, that here we call $\mathcal{E}^{(D)}$ to be consistent with our prevoius notation.

  \begin{proposition}
   Let $W_1^{\dep} \in \cW(\sigma_1)$ be the normalised $p$-depleted vector (so that $W_1(\stbt{x}{}{}{1}) = \operatorname{ch}_{\Zp^\times}(x)$). Then for any $\ell \ge 2$ we have
   \[ \fz^{\mathrm{sph}}\left(\hat\tau W'_{\alpha, \Si}[\ell], W_1^{\dep}, 
    W'_{\mathfrak{a}_2}[\ell]\right) = \frac{p^3}{(p+1)^2(p-1)} \mathcal{E}^{(D)}.\]
  \end{proposition}

  In an ideal world (i.e.~if we had a comprehensive version of ``Siegel-parabolic higher Hida theory'' available to us), the above formula would presumably be the right one to use for the interpolation property of our $p$-adic $L$-functions. However, for technical reasons we are constrained to work at Iwahori level, so we need a variant of this formula.

  \begin{proposition}
   For any $\ell \ge 2$ we have
   \[
    \fz^{\mathrm{sph}}\left(w_{01}^{-1} \tau w_2 \cdot W^{\prime, \Iw}_{\alpha, \beta}[\ell], \stbt{p^\ell}{}{}{1} W_1^{\dep}, W_{\mathfrak{a}_2}\right) = \left(\tfrac{p^2}{\beta \mathfrak{b}_2}\right)^t \cdot \frac{p^3}{(p+1)^2(p-1)} \cdot \mathcal{E}^{(D)}.
   \]
  \end{proposition}

  \begin{proof}
   Since $w_2 = w_1\cdot w_{\Si}$, where $w_{\Si}$ is the long Weyl element of $M_{\Si}$, we can calculate this quantity as
   \[ \fz^{\mathrm{sph}}\left( \hat{\tau} \cdot w_{\Si} W^{\prime, \Iw}_{\alpha, \beta}[\ell], \stbt{p^\ell}{}{}{1} W_1^{\dep}, w W_{\mathfrak{a}_2}\right). \]
   Everything in sight is invariant under the principal congruence subgroup mod $p^\ell$, so we may apply \cref{prop:stabilises1} to express this as
   \[ \Delta \cdot \mathfrak{y}^{\sph}\left(w \cdot W'_{\beta/p},\stbt{p^\ell}{}{}{1} W_1^{\dep}, w W_{\mathfrak{a}_2}\right), \]
   where both $w$'s denote $\stbt{}{1}{-1}{} \in \GL_2$, and $W'_{\gamma/p}$ is a $U'$-eigenvector of eigenvalue $\beta/p$ in $\cW(\sigma_0)$. We have
   We have $w W_{\mathfrak{a}_2} = \fb_2^{-\ell} \stbt{p^\ell}{}{}{1} W'_{\fa_2}[\ell]$; and similarly
   $w \cdot W'_{\beta/p}[\ell] = \left(\tfrac{p^2}{\beta}\right)^\ell W_{\beta/p}$. So we obtain
   \[ \Delta \cdot \left(\tfrac{p^2}{\beta\fb_2}\right)^\ell \mathfrak{y}^{\sph}\left(\stbt{p^\ell}{}{}{1} \cdot W_{\beta/p},\stbt{p^\ell}{}{}{1} W_1^{\dep}, \stbt{p^\ell}{}{}{1} W'_{\mathfrak{a}_2}[\ell]\right).\]
   After cancelling the $\stbt{p^\ell}{}{}{1}$ terms using the $\GL_2$-equivariance, we can now conclude using \cref{prop:relateintegrals2}.
  \end{proof}

  \begin{remark}\label{improved}
   If, in place of $W_1^{\dep}$, we use the Iwahori eigenvector $W_{\fa_2}$, then this corresponds to deleting the degree-one factor $\left(1 - \frac{p^2}{\gamma \fb_1 \fa_2}\right)$ from $\mathcal{E}^{(D)}$. This gives exactly the ``greatest common divisor'' of $\mathcal{E}^{(D)}$ and $\mathcal{E}^{(E)}$.
  \end{remark}

\let\MR\undefined
\newlength{\bibitemsep}
\setlength{\bibitemsep}{0.75ex plus 0.05ex minus 0.05ex}
\newlength{\bibparskip}
\setlength{\bibparskip}{0pt}
\let\oldthebibliography\thebibliography
\renewcommand\thebibliography[1]{%
 \oldthebibliography{#1}%
 \setlength{\parskip}{\bibparskip}%
 \setlength{\itemsep}{\bibitemsep}%
}
\providecommand{\noopsort}[1]{\relax} 
 \bibliographystyle{../amsalphaurl}
 \bibliography{../references}

\end{document}